\documentclass[11pt, reqno, a4paper, final]{amsart}

\usepackage[utf8]{inputenc}      
\usepackage[english]{babel}
\usepackage[T1]{fontenc}
\usepackage{amsmath,amsfonts,amssymb,amsthm,amscd}
\usepackage{eucal}
\usepackage{array}
\usepackage{mathtools}
\usepackage{color}
\usepackage{bbm}
\usepackage{graphicx}
\usepackage{hyperref}
\usepackage{a4wide}										 	
\usepackage{mathrsfs}									  
\usepackage[all]{xy}										
\usepackage{url}												
\usepackage{verbatim,epsfig}
\usepackage[notref,notcite,draft]{showkeys}   
\usepackage{xfrac} 
\usepackage[linewidth=1pt]{mdframed}

\newcounter{claim-counter}

\theoremstyle{plain}

\newtheorem{thm}{Theorem}[section] 
\newtheorem{theorem}{Theorem}[section]

\newtheorem{lemma}[thm]{Lemma}
\newtheorem{prop-defi}[thm]{Definition \& Proposition}

\newtheorem*{thm*}{Theorem}
\newtheorem*{prop*}{Proposition}
\newtheorem*{cor*}{Corollary}
\newtheorem{proposition}[thm]{Proposition}

\usepackage{thmtools}
\declaretheorem[style=theorem,name={Theorem}]{theoremletter}

\theoremstyle{definition}

\newtheorem{definition}[thm]{Definition}

\newtheorem{remark}[thm]{Remark}

\newtheorem*{claim*}{Claim}

\newcommand{\NN}{{\mathbb N}}
\newcommand{\ZZ}{{\mathbb Z}}

\newcommand{\RR}{{\mathbb R}}
\newcommand{\CC}{{\mathbb C}}

\newcommand{\varps}{{\varepsilon}}

\renewcommand{\Im}{{\operatorname{Im}}}
\renewcommand{\Re}{{\operatorname{Re}}}

\newcommand{\To}{\longrightarrow}

\newcommand{\del}{{\partial}}

\renewcommand{\min}{{\operatorname{min}}}

\renewcommand{\leq}{\leqslant}
\renewcommand{\geq}{\geqslant}

\newcommand{\Dom}{\operatorname{Dom}}
\newcommand{\Lip}{\operatorname{Lip}}

\newcommand{\black}{\color{black}}

\renewcommand{\epsilon}{{\varepsilon}}

\allowdisplaybreaks

\title{Gromov-Hausdorff convergence of quantised intervals}

\author{Thomas Gotfredsen}
\address{Thomas Gotfredsen, Department of Mathematics and Computer Science, University of Southern Denmark, Campusvej 55, DK-5230 Odense M, Denmark}
\email{thgot@imada.sdu.dk}

\author{Jens Kaad}
\address{Jens Kaad, Department of Mathematics and Computer Science, University of Southern Denmark, Campusvej 55, DK-5230 Odense M, Denmark}
\email{kaad@imada.sdu.dk}

\author{David Kyed}
\address{David Kyed, Department of Mathematics and Computer Science, University of Southern Denmark, Campusvej 55, DK-5230 Odense M, Denmark}
\email{dkyed@imada.sdu.dk}


\begin{document}

\begin{abstract}
The Podle\'s quantum sphere $S^2_q$ admits a natural commutative $C^*$-subalgebra $I_q$  with spectrum $\{0\} \cup \{q^{2k}: k\in \NN_0\}$, which may therefore be considered as a quantised version of a classical interval. We study here the compact  quantum metric space structure  on $I_q$ inherited from the corresponding structure on $S^2_q$, and provide an explicit formula for the metric induced on the spectrum. Moreover, we show that the resulting metric spaces vary continuously in the deformation parameter $q$ with respect to  the Gromov-Hausdorff distance, and that they converge to a classical interval of length $\pi$ as $q$ tends to $1$.
\end{abstract}

\maketitle

\section{Introduction}
The study of compact quantum metric spaces dates back to the work of Connes \cite{Connes}, in which he studied metrics on state spaces of spectral triples. This notion was later formalised in the works of Rieffel \cite{Rieffel-1,Rieffel-2,Rieffel-4}, in which the weak $*$-topology on the state space is metrised by the Monge-Kantorovich metric coming from a so-called Lip-norm on a $C^*$-algebra (see Section \ref{sec:prelim} for details).  As shown by Rieffel, the classical Gromov-Hausdorff distance admits an analogue, known as quantum Gromov-Hausdorff distance, for compact quantum metric spaces, and this notion was later refined by   Latrémolière through his notion of propinquity \cite{Latremoliere}. Although examples of compact quantum metric spaces are abundant,  some of the most basic examples from non-commutative geometry are not well understood from this point of view, and only very recently,  Aguilar and Kaad \cite{Kaad-Aguilar} showed that the Podle\'s standard sphere $S_q^2$, introduced as a homogeneous space of Woronowicz' $q$-deformed $SU(2)$ \cite{Podles,Woronowicz-SU2}, admits a natural compact quantum metric space structure stemming from its non-commutative geometry. More precisely, Aguilar and Kaad show that the Lip-norm arising from the Dirac operator  ${D_q}$ of the Dąbrowski-Sitarz spectral triple \cite{Dabrowski-Sitarz}, does indeed provide a quantum metric structure on $S_q^2$. The main question left open in \cite{Kaad-Aguilar} is that  of quantum Gromov-Hausdorff convergence of $S_q^2$ to the classical $2$-sphere $S^2$ as the deformation parameter tends to $1$. This question seems rather difficult to settle\footnote{Note: building on the results of the present paper, this was very recently settled in \cite{AKK:PodCon}.}, and the aim of the present paper is to show that the Podle\'s sphere $S_q^2$ contains a natural commutative $C^*$-algebra $I_q$ for which the corresponding convergence question can be settled, and that the answer supports the more general conjecture that $S_q^2$ converges to $S^2$ as $q$ tends to $1$.  The Podle\'s sphere is generated by a self-adjoint operator $A$ and a non-normal operator $B$ (see Section \ref{sec:prelim} for precise definitions), and the $C^*$-algebra $I_q$ is simply the unital $C^*$-algebra generated by $A$ inside $S_q^2$. 
Since $S_q^2$ admits a rather accessible representation on $B(\ell^2(\NN_0))$ \cite[Proposition 4]{Podles}, the spectrum of the self-adjoint generator $A\in S_q^2$ is easily derivable,  and one finds that for  $q\in(0,1)$ this is exactly the set

$$X_q=\lbrace 0 \rbrace \cup \lbrace q^{2k} \colon k\in \NN_0 \rbrace,$$
which can therefore be viewed as a quantised version of a classical interval. The Lip-norm $L_{D_q}$ coming from the Dirac operator on $S_q^2$ therefore, in particular, provides a metric on the state space of $I_q\cong C(X_q)$ and embedding $X_q$ into the state space of $C(X_q)$ as point-evaluations, we obtain a metric $d_q$ on $X_q$. Our first main result determines an explicit formula for this metric.

\begin{theoremletter}
\label{Theorem:Lip-Lip}
For $q\in (0,1)$, the metric $d_q$ on $X_q$ is given by the following formula:
\begin{align*}
d_q(x,y)\coloneqq
\begin{cases} \hspace{0.4cm }0 &\mbox{if } x=y \\
 \displaystyle\sum_{k=\min\lbrace m,n \rbrace}^{\max\lbrace m,n \rbrace-1} \frac{(1-q^2)q^k}{\sqrt{1-q^{2(k+1)}}} & \mbox{if } x= q^{2n} \text{ and } y= q^{2m}  \text{ with } n\neq m \\
 \displaystyle\sum_{k=n}^\infty \frac{(1-q^2)q^k}{\sqrt{1-q^{2(k+1)}}} & \text{ if } x= q^{2n} \text{ and } y=0 \text{ or } x=0 \text{ and } y= q^{2n}. 
  \end{cases}
\end{align*}
\end{theoremletter}
 
When $q=1$, the spectrum of the operator $A$ becomes $X_1:=[0,1]$ and in Section \ref{subsec:commutative} we will show that when $X_1$ is equipped with the metric $d_1$  inherited from the classical $2$-sphere $S^2$, then the space $(X_1,d_1)$ becomes isometrically isomorphic to $[-\frac{\pi}{2}, \frac{\pi}{2}]$ with its standard Euclidian metric.  Our second main theorem therefore confirms that the quantised intervals do indeed converge to the appropriate classical interval as the deformation parameter tends to 1:

\begin{theoremletter}
\label{Theorem:Cont}
The metric spaces $(X_q,d_q)$ vary continuously with respect to the Gromov-Hausdorff distance in the deformation parameter $q\in (0,1)$ and converge to  
the interval $\left[-\tfrac{\pi}{2},\tfrac{\pi}{2}\right]$ with its standard metric as $q$ tends to 1.

\end{theoremletter}

On the class of commutative compact quantum metric spaces, convergence in both  Latrémolière's propinquity \cite{Latremoliere} and Rieffel's quantum Gromov-Hausdorff distance \cite{Rieffel-distance} is implied by convergence in classical Gromov-Hausdorff distance (see Remark \ref{Rem:Dist}) and Theorem \ref{Theorem:Cont} therefore settles all the natural convergence question for the algebras $I_q\cong C(X_q)$.\\

The paper is structured as follows: In the first part we introduce the basic definitions concerning quantum metric spaces, Gromov-Hausdorff distance, $SU_q(2)$ and the standard Podle\'s sphere and the associated Dąbrowski-Sitarz spectral triple. In the second part we first give a description of $I_q$ in the continuum case, i.e. when $q=1$, followed by a thorough treatment of the quantised case, where $SU(2)$ is deformed by a parameter $q\in(0,1)$. For this we provide a detailed treatment of the metric $d_q$, on $X_q$ and its Lipschitz semi-norm from which we can prove Theorem \ref{Theorem:Lip-Lip}, and finally we use this to prove Theorem \ref{Theorem:Cont}.

\subsubsection*{Acknowledgments}
The authors gratefully acknowledge the financial support  from the Independent Research Fund Denmark  through grant no.~7014-00145B and grant no.~9040-00107B.
We are also grateful for the comments and suggestions provided by the anonymous referee.
\subsubsection*{Standing conventions}
The semi-norms appearing in this text are defined everywhere on unital $C^*$-algebras and may take the value infinity.


\section{Preliminaries}\label{sec:prelim}
\subsection{Quantum metric spaces}
We begin this section by recalling some basic facts about metric spaces.
Let $(X,d)$ be a compact metric space. The \emph{Lipschitz semi-norm}, $L_{d}\colon C(X)\to [0,\infty]$, on $C(X)$ is defined by the formula
 $$L_{d}(f)\coloneqq \sup \left\lbrace \frac{\vert f(x)-f(y)\vert}{d(x,y)} \, : x \neq y \right\rbrace; \quad f\in C(X).$$
A continuous function $f : X \to \CC$ is then said to be a Lipschitz function when $L_d(f) < \infty$ and in this case $L_d(f)$ agrees with the Lipschitz constant. The Lipschitz functions on $X$ form a $*$-subalgebra which we denote by $C_{\Lip}(X) \subset C(X)$. Given subsets $A,B\subset X$, their Hausdorff-distance is defined as
\[
\operatorname{dist}_H^d(A,B)\coloneqq\inf \lbrace r\geq 0 \vert A\subset \mathbb{B}(B,r) \text{ and } B\subset \mathbb{B}(A,r)\rbrace,
\]
where $\mathbb{B}(A,r)$ denotes the set $\{x\in X : \exists a\in A : d(x,a)<r\}.$
For two metric spaces $(X,d_X),(Y,d_Y)$,  their \emph{Gromov-Hausdorff distance}  is defined  as
\begin{align*}
\operatorname{dist}_{{GH}}(X,Y) = \inf\lbrace \operatorname{dist}_H^{d_Z}(\iota_X(X), \iota_Y(Y) \rbrace,
\end{align*}
where the infimum ranges over all metric spaces $(Z,d_Z)$ and all isometric embeddings $\iota_X\colon X\to Z$ and $\iota_Y\colon Y\to Z$.
Next, we will recall the relevant definitions for quantum metric spaces.
\begin{definition}[\cite{Rieffel-1, Rieffel-2, Rieffel-4}]
\label{Def:QMS}
Let $A$ be a unital $C^*$-algebra, and let $ L\colon A \to [0,\infty]$ be a semi-norm. We say that $(A,L)$ is a \emph{compact quantum metric space}, and that $L$ is a \emph{Lip-norm}, if the following conditions are satisfied:
\begin{enumerate}
\item $\Dom(L)\coloneqq \lbrace a\in A : L(a)<\infty	\rbrace$ is dense in $A$;
\item $L$ is $*$-invariant and lower semi-continuous on $A$;
\item $\ker(L)\coloneqq \lbrace a\in A : L(a)=0\rbrace =\CC 1_A$;
\item The Monge-Kantorovich metric on the state space $S(A)$ of $A$, given by
$$\operatorname{mk}_L(\mu,\nu)\coloneqq \sup\lbrace \vert \mu(a)-\nu(a)\vert \, : a \in A,L(a)\leq 1 \rbrace, \quad \text{for }\mu,\nu\in S(A)$$ metrises the weak $*$-topology.
\end{enumerate}
\end{definition}
The model example for a compact quantum metric space is, unsurprisingly, $(C(X),L_{d})$ where $(X,d)$ is a compact metric space. In this case it is a well-known fact that the Monge-Kantorovich metric recaptures the metric $d$ on $X$ when the latter is viewed as a subset of the state space of $C(X)$:
\[
d(x,y)=\sup\{|f(x)-f(y)| : f\in C(X), L_d(f)\leq 1\}.
\]
Another interesting class of examples, which dates back to the work of Connes \cite{Connes}, comes from certain spectral triples:
the setting is thus that of a separable Hilbert space $H$ with  a self-adjoint densely defined operator $D\colon \Dom(D)\to H$, and a unital $C^*$-algebra  $A$ represented on $H$ via a $*$-homomorphism $\rho\colon A \to B(H)$.
Then one can  define the \emph{Lipschitz algebra} $\Lip_D(A)$, to consist of all elements $x\in A$ which preserve $\Dom(D)$, and for which $[D,\rho(x)]\colon \Dom(D) \to H$ admits a bounded extension to $H$, which will be denoted by $\partial(x)\in B(H)$. Clearly, $\Lip_D(A) \subset A$ is a $*$-subalgebra and it follows from the definition of a spectral triple that $\Lip_D(A) \subset A$ is norm-dense. From the spectral triple $(A,H,D)$, we also obtain a semi-norm as follows:
\begin{definition}
Define $L_D\colon A\to [0,\infty]$ by the formula
$$L_D(x)\coloneqq \sup \left\{ \left\vert \left\langle\xi,\rho(x^*)D\eta\right\rangle - \left\langle \rho(x)D\xi,\eta \right\rangle \right\vert : {\xi,\eta \in \Dom(D), \Vert \xi\Vert = \Vert \eta \Vert=1} \right\} .$$
\end{definition}
A first result says that $x\in \Lip_D(A)$ exactly when $L_D(x)$ is finite, and in this case $L_D(x)=\Vert \partial(x)\Vert$, see e.g. \cite[Lemma 2.3]{Kaad-Aguilar}. Moreover, $L_D : A \to [0,\infty]$ is lower semi-continuous and $*$-invariant, see \cite[Proposition 3.7]{Rieffel-2}.
The above construction does in general not yield a quantum metric space, but due to the work of Rieffel, there are tools available for verifying whether or not this is the case (see for instance \cite[Theorem 1.8]{Rieffel-1}).\\

Quantum analogues of the Gromov-Hausdorff distance have been defined by Rieffel and Latrémolière, and we refer the reader to \cite{Rieffel-distance,Latremoliere} for concrete definitions. For our purposes, it suffices to know that when the compact quantum metric spaces in question are of the form $(C(X),L_{d})$, then both analogues are dominated by the classical Gromov-Hausdorff distance, see Remark \ref{Rem:Dist}.
\\

\subsection{The Standard Podle\'s Sphere}
The central object of interest in this paper is the standard Podle\'s quantum sphere, which is defined as a particular $C^*$-subalgebra of Woronowicz' \cite{Woronowicz-SU2} quantum group $SU_q(2)$ as given below.
Fix $q\in(0,1]$, and let $SU_q(2)$ denote the universal unital $C^*$-algebra with generators $a$ and $b$ defined such that the following relations are satisfied:

\begin{align*}
ba=qab, \quad b^*a&=qab^*, \quad bb^* =b^*b
\\
a^*a+ q^2bb^* = 1& = aa^*+bb^* .
\end{align*}
We denote the unital $*$-subalgebra generated by $a$ and $b$ by $\mathcal{O}(SU_q(2))$, and
by $\mathcal{O}(S_q^2)$  the unital $*$-subalgebra of $\mathcal{O}(SU_q(2))$ generated by the elements
$$A\coloneqq b^*b \quad \text{and} \quad B\coloneqq ab^*.$$
The standard Podle\'s quantum sphere, $S_q^2$, is defined as the norm-closure of $\mathcal{O}(S^2_q)\subset SU_q(2)$ \cite{Podles}.
We remark that from the defining relations of $SU_q(2)$ we obtain a similar set of relations for $A$ and $B$:
\begin{align*}
AB=q^2BA, &\quad A=A^*
\\
BB^*=q^{-2}A(1-A), &\quad B^*B=A(1-q^2A).
\end{align*}
The $C^*$-algebra $SU_q(2)$ comes equipped with a natural faithful state, called the Haar state, which we denote by $h : SU_q(2) \to \CC$, see e.g. \cite[Section 11.3.2]{Klimyk}. We let $L^2(SU_q(2))$ denote the separable Hilbert space obtained by applying the GNS-construction to the $C^*$-algebra $SU_q(2)$ equipped with the Haar state.

From now on, we assume that $q \neq 1$. Define an automorphism $\partial_k$ on $\mathcal{O}(SU_q(2))$ by $\partial_k(x)=q^{\frac{1}{2}}x$ if $x\in\lbrace a,b\rbrace$, and $\partial_k(x)=q^{-\frac{1}{2}}x$ if $x\in\lbrace a^*,b^*\rbrace$, and for each $n\in \ZZ$, define the vector subspaces $$\mathcal{A}_n\coloneqq \lbrace x\in \mathcal{O}(SU_q(2)) : \partial_k(x)=q^{n/2}x\rbrace \subset \mathcal{O}(SU_q(2)).$$ It turns out that $\mathcal{A}_0 = \mathcal{O}(S_q^2)$ and that the algebra structure on $\mathcal{O}(SU_q(2))$ allows us to consider each $\mathcal{A}_n$ as a left module over $\mathcal{O}(S_q^2)$. We let $H_+$ and $H_-$ denote the separable Hilbert spaces obtained by taking the Hilbert space closures of $\mathcal{A}_1$ and $\mathcal{A}_{-1}$ (respectively) when considered as subspaces of $L^2(SU_q(2))$. The GNS-representation of $SU_q(2)$ on $L^2(SU_q(2))$ (when properly restricted) then provides us with two unital $*$-homomorphisms $\rho_+ : S_q^2 \to B(H_+)$ and $\rho_- : S_q^2 \to B(H_-)$. 

By \cite{Dabrowski-Sitarz} there exists an even spectral triple, $(S_q^2,H_+ \oplus H_-,D_q)$, where the representation in question is given by the direct sum $\rho : \rho_+ \oplus \rho_- : S_q^2 \to B(H_+ \oplus H_-)$. For an explicit construction of the Dirac operator $D_q : \Dom(D_q) \to H_+ \oplus H_-$, we refer to \cite{Dabrowski-Sitarz,Neshveyev-Tuset} or \cite{Kaad-Aguilar}.

For $x\in \Lip_{D_q}(S_q^2)$, the associated operator $\partial(x)$ (obtained as the closure of $[D_q,\rho(x)]$) takes the form
\[
\begin{pmatrix}
 0 & \partial_2(x)
 \\
 \partial_1(x) & 0 
 \end{pmatrix}\colon H_+\oplus H_-\to H_+\oplus H_-,
\]
where $\partial_1\colon \Lip_{D_q}(S_q^2)\to B(H_+,H_-)$ and $\partial_2\colon \Lip_{D_q}(S_q^2)\to B(H_-, H_+)$ are derivations satisfying $\partial_2(x^*)=-\partial_1(x)^*$ (remark in this respect that $B(H_+,H_-)$ and $B(H_-,H_+)$ can be considered as bimodules over $S_q^2$ via the representations $\rho_+$ and $\rho_-$). Consequently the Lip-norm is, for $x\in \Lip_{D_q}(S_q^2),$ given by
\[
L_{D_q}(x)= \max \left\lbrace\Vert \partial_1(x)\Vert, \Vert \partial_1(x^*)\Vert \right\rbrace.
\]
  
By \cite[Proposition 4]{Podles}, $S_q^2$ admits a faithful representation, $\pi\colon S_q^2 \to B(\ell^2(\NN_0))$, defined by
\begin{equation}
\label{Eq:PodRep} \pi(A)(e_k)\coloneqq q^{2k}e_k, \quad  \pi(B)(e_k)=q^k \sqrt{1-q^{2(k+1)}}e_{k+1},
\end{equation}
where $e_k$ denotes the characteristic function on the point-set $\lbrace	k\rbrace\subset\NN_0$. In fact, this representation even provides a $*$-isomorphism to the unitisation of the compact operators on $\ell^2(\NN_0)$.
Using this representation it is easy to see that the spectrum of the operator $A$ for a specific $q\in(0,1)$ is given by $$X_q:=\lbrace 0 \rbrace  \cup \lbrace q^{2k} : k \in \NN_0\rbrace.$$
Hence the indicator functions $\chi_{\lbrace q^{2k} \rbrace} \colon X_q\to \lbrace 0,1\rbrace$ are continuous for all $k$. In fact, these indicator functions and the unit generate $C(X_q)$, since any continuous function, $f \colon X_q \to \CC$, can be written as $f(0) + \sum_{k=0}^\infty  (f(q^{2k})-f(0)) \cdot \chi_{\lbrace q^{2k}\rbrace}$, where $\lim_{k \to \infty} f(q^{2k}) = f(0)$.
By \cite[Theorem 8.3]{Kaad-Aguilar}, $(S_q^2,L_{D_q})$ is a compact quantum metric space, and consequently so is $I_q:=C^*(A,1)\cong C(X_q)$ with the restricted Lip-norm. The compact quantum metric space $(I_q,L_{D_q})$ is our main object of interest in the present paper.  As $I_q$ is commutative, the Lip-norm $L_{D_q}$ defines a genuine metric $d_q$ on $X_q$ when the latter is considered as a subset of the state space $S(S_q^2)$. In order to describe $d_q$ explicitly, the following lemma will be key:
\begin{lemma}[{\cite[Lemma 5.3]{Kaad-Aguilar}}]
\label{Lemma:Char}
Let $k\in \NN_0$ and let $q \in (0,1)$. We have that $\chi_{\lbrace q^{2k}\rbrace}(A) \in \Lip_{D_q}(S_q^2)$ and the derivative is given by
\begin{align*}
\partial_1(\chi_{\lbrace q^{2k} \rbrace}(A))=
\frac{1}{q^{2k}(1-q^2)}\chi_{\lbrace q^{2k}\rbrace}(A)\cdot b^*a^*
-\frac{1}{q^{2(k-1)}(1-q^2)}\chi_{\lbrace q^{2(k-1)}\rbrace}(A)\cdot b^*a^*  
\end{align*}
In particular, we obtain that 
\begin{equation}\label{eq:qdiff}
\partial_1(f(A)) = \sum_{k = 0}^\infty \frac{f(q^{2k}) - f(q^{2(k+1)})}{q^{2k}(1-q^2)} \chi_{\lbrace q^{2k}\rbrace}(A)\cdot b^* a^* 
\end{equation}
for every $f \in \operatorname{span}_\CC\{\chi_{\lbrace q^{2k}\rbrace} : k\in \NN_0\}$.
\end{lemma}

\begin{remark}
The formula in \eqref{eq:qdiff} for $\partial_1(f(A))$ is related to the notion of $q$-differentiation from $q$-calculus. Indeed, the $q^2$-differentiation of $f \in \operatorname{span}_\CC\{\chi_{\lbrace q^{2k}\rbrace} : k\in \NN_0\}$ would be given by
\[
\mathcal{D}_{q^2}(f) = \sum_{k = 0}^\infty \frac{f(q^{2k}) - f(q^{2(k+1)})}{q^{2k}(1-q^2)} \chi_{\lbrace q^{2k}\rbrace} ,
\]
see for example \cite[Chapter 2.2]{Klimyk}. The extra term $b^* a^*$ appearing in \eqref{eq:qdiff} comes from the geometry of the quantised $2$-sphere as it operates between the Hilbert space completions $H_+$ and $H_-$ of the quantised spinor bundles $\mathcal{A}_1$ and $\mathcal{A}_{-1}$.
\end{remark}


\section{Metric Properties of the Quantised Interval}
In this section we first provide the explicit descriptions of the compact metric spaces $(X_q,d_q)$ which encode the compact quantum metric space structure of $(I_q,L_{D_q})$. More precisely, the algebra of Lipschitz functions of the metric space $(X_q,d_q)$ must agree with the Lipschitz algebra $\Lip_{D_q}(S_q^2)\cap I_q$  and the two semi-norms must agree, in the  sense that $L_{D_q}(f(A))=L_{d_q}(f)$ whenever $f$ is a Lipschitz function on $(X_q,d_q)$. This analysis is separated into the case $q=1$, referred to as the continuum case, and the case $q<1$, referred to as the quantised case.

\subsection{The continuum case}\label{subsec:commutative}
We consider the $2$-sphere $S^2 = \{ (x_1,x_2,x_3) \in \RR^3 : x_1^2 + x_2^2 + x_3^2 = 1 \}$ whereas $S^3 = \{ (z,w) \in \CC^2 : |z|^2 + |w|^2 = 1 \}$ both equipped with the subspace topology coming from the usual topology on $\RR^3$ and $\CC^2$.

In the situation where $q = 1$ we have a homeomorphism between the characters of $SU_q(2)$ and the $3$-sphere $S^3$, which sends $(z,w)\in S^3\subset \CC^2$ to the unique character $\chi_{z,w}$ satisfying that $\chi_{z,w}(a)=z$ and $\chi_{z,w}(b)=w$ (see \cite{Woronowicz-SU2}). Consequently, we can identify $SU_q(2)$ with $C(S^3)$ such that $a(z,w)= z$ and $b(z,w) = w$. We may moreover view the $2$-sphere $S^2$ as the quotient space of $S^3$ under the circle action $\lambda \cdot (z,w) := (\lambda \cdot z,\lambda \cdot w)$ and this identification happens via the Hopf-fibration
\[
S^3 \ni (z,w)\longmapsto \left(2\Re(z\bar{w}),2\Im(z\bar{w}),|z|^2-|w|^2\right)\in S^2.
\]
Since both $A(z,w) = (b^* b)(w) = |w|^2$ and $B(z,w) = z \bar{w}$ are invariant under the circle action we may consider them as continuous function on $S^2$ and as such they are given by
\[
A(x_1,x_2,x_3) = \frac{1 - x_3}{2} \qquad \mbox{and} \qquad B(x_1,x_2,x_3) = \frac{x_1 + i x_2}{2} .
\]
It is now clear that $A$ has range $[0,1]$ and so we have a $*$-isomorphism $C([0,1])\cong I_1$. Let $d_1$ be the metric on $[0,1]$ obtained from the standard round metric on $S^2$ so that
\[
d_1(s,t) := \inf\big\{ d_{S^2}\big( (x_1,x_2, 1-2s), (y_1,y_2, 1-2t) \big) : x_1^2 + x_2^2 + (1-2s)^2 = 1 = y_1^2 + y_2^2 + (1-2t)^2 \big\}
\]
for all $s,t \in [0,1]$. We record the following elementary result:  
\begin{proposition}
 \label{Prop:comm}
The map $\phi\colon [-\tfrac{\pi}{2},\tfrac{\pi}{2}]\to[0,1]$ given by $\phi(t)=\tfrac{1}{2} + \tfrac{1}{2}\sin(t)$ is an isometric isomorphism when $[-\tfrac{\pi}{2},\tfrac{\pi}{2}]$ is equipped with the standard Euclidean metric $d$ and $[0,1]$ is equipped with the metric $d_1$. In particular, we have a $*$-isomorphism $\beta\colon C([-\tfrac{\pi}{2},\tfrac{\pi}{2}])\to I_1$, $\beta(f) = (f\circ \phi^{-1})(A)$, which maps $C_{\Lip}([-\tfrac{\pi}{2},\tfrac{\pi}{2}])$ onto $I_1\cap C_{\Lip}(S^2)$ and satisfies $L_{d_{S^2}}(\beta(f))=L_{d}(f)$.
\end{proposition}

\begin{remark}
For completeness, we note that when $q=1$, the standard Podle\'s sphere is of course isomorphic to $C(S^2)$. Indeed, the continuous maps corresponding to $A$ and $B$ separate points in $S^2$ and the Stone-Weierstrass Theorem then shows that $S_1^2=C^*(1,A,B) \cong C(S^2)$.
\end{remark}

\subsection{The quantised case}
We will now address the case of a fixed $q\in (0,1)$. We let $X_q$ denote the spectrum of $A\in S_q^2$, and, as we already saw,  $X_q=\{0\}\cup \{q^{2k} : k\in \NN_0\}$. As explained in the introduction, the Lip-norm $L_{D_q}$ gives rise to a metric on the state space of $C^*(A,1)\cong C(X_q)$, which therefore, in particular, determines a metric $d_q$ on $X_q$ when the latter is  viewed as a subset of the state space via point evaluations.
The aim of the current section is to find an explicit formula for this metric, and show that the metric spaces $(X_q,d_q)$ converge in the Gromov-Hausdorff distance to the Euclidean interval $[-\tfrac{\pi}{2},\tfrac{\pi}{2}]$ as $q$ tends to $1$. 

To this end, we consider the function $\rho_q\colon [-1,\infty) \to \RR$ by
\[
\rho_q(x)\coloneqq \frac{\sqrt{1-q^{2(x+1)}}}{(1-q^2)q^x} .
\]

\begin{definition}\label{d:defmetric} Define the metric $d_q\colon X_q\times X_q \to [0,\infty)$ by
\begin{align*}
d_q(x,y)\coloneqq
\begin{cases} \hspace{0.4cm }0 &\mbox{if } x=y \\
 \displaystyle\sum_{\min\lbrace m,n \rbrace}^{\max\lbrace m,n \rbrace-1} \frac{1}{\rho_q(k)} & \mbox{if } x=q^{2n} \text{ and } y=q^{2m}  \text{ with } n\neq m \\
 \displaystyle\sum_{k=n}^\infty\frac{1}{\rho_q(k)} & \text{ if } x=q^{2n} \text{ and } y=0 \text{ or } x=0 \text{ and } y=q^{2n}. 
  \end{cases}
\end{align*}
\end{definition}
Remark that the series $\sum_{k=0}^\infty \frac{1}{\rho_q(k)}$ is convergent as can be seen from the estimate 
\begin{equation}\label{eq:estrho}
\frac{1}{\rho_q(k)} = \frac{q^k (1-q^2)}{\sqrt{1 - q^{2(k+1)}}} \leq q^k  \qquad \mbox{for all } k \in \NN_0 . 
\end{equation}
%

In order to prove Theorem \ref{Theorem:Lip-Lip}, we need several lemmas, the first of which shows that the Lipschitz semi-norm on $C(X_q)$ defined by the metric $d_q$ and the Lip-norm $L_{D_{q}}$ on $I_q$ agree on all finite linear combinations of characteristic functions on $X_q$:

\begin{lemma}\label{Lemma:Linearcomb}
For any $f\in \operatorname{span}_\CC\{\chi_{\lbrace q^{2k}\rbrace} : k\in \NN_0\} \subset C(X_q)$, it holds that $f(A) \in \Lip_{D_q}(S_q^2) \cap I_q$. Moreover, we have the identities
\[
L_{D_q}(f(A))= \max\lbrace \rho_q(k) \cdot \vert f(q^{2k})-f(q^{2(k+1)}) \vert : k \in \NN_0 \rbrace = L_{d_q}(f).
\]
In particular, $f$ is also Lipschitz with respect to the metric $d_q$.
\end{lemma}
Note that the maximum is indeed well-defined, since $f$ is non-zero at no more than finitely many elements from $X_q$.
\begin{proof}
Let $f\in \operatorname{span}_\CC\{\chi_{\lbrace q^{2k}\rbrace} : k\in \NN_0\}$ be given. The fact that $f(A) \in \Lip_{D_q}(S_q^2) \cap I_q$ is a consequence of Lemma \ref{Lemma:Char}. Moreover, from Lemma \ref{Lemma:Char} and the defining identities for $SU_q(2)$ we obtain that
\[
\begin{split}
\partial_1( f(A))\partial_1(f(A))^*
& = A(1-q^2A) \sum_{k = 0}^\infty \frac{\vert f(q^{2k}) - f(q^{2(k+1)}) \vert^2}{ q^{4k}(1 - q^2)^2} \chi_{\lbrace q^{2k}\rbrace}(A) \\
& = \sum_{k = 0}^\infty \rho_q(k)^2 \cdot \vert f(q^{2k}) - f(q^{2(k+1)}) \vert^2 \chi_{\lbrace q^{2k}\rbrace}(A) .
\end{split}
\]
The continuous functional calculus applied to $A \in I_q$ then implies that
\begin{equation}\label{eq:paone}
\big\| \partial_1(f(A)) \big\|^2
= \max\lbrace \rho_q(k) \cdot \vert f(q^{2k})-f(q^{2(k+1)}) \vert : k \in \NN_0 \rbrace .
\end{equation}
The identity
\[
L_{D_q}(f(A)) = \max\lbrace \rho_q(k) \cdot \vert f(q^{2k})-f(q^{2(k+1)}) \vert : k \in \NN_0 \rbrace
\]
now follows since the formula in \eqref{eq:paone} implies that $\| \partial_2(f(A)) \| = \| \partial_1(\overline{f}(A)) \| = \| \partial_1(f(A)) \|$.

For the second equality, choose $l\in \NN_0$ such that 
\[
\rho_q(l) \cdot \vert f(q^{2l})- f(q^{2(l+1)}) \vert 
= \max \lbrace \rho_q(k) \cdot \vert f(q^{2k})-f(q^{2(k+1)}) \vert : k \in \NN_0 \rbrace .
\]
This choice of $l \in \NN_0$ implies that
\[
\vert f(q^{2k})-f(q^{2(k+1)}) \vert \leq \vert f(q^{2l})- f(q^{2(l+1)}) \vert \cdot \frac{ \rho_q(l) }{\rho_q(k)}
\]
for all $k \in \NN_0$. Thus, for every $m<n$ we may now estimate as follows:
\begin{equation}\label{eq:finest}
\begin{split}
\vert f(q^{2m}) - f(q^{2n}) \vert & \leq \sum_{k = m}^{n-1} \vert f(q^{2k}) - f(q^{2(k+1)}) \vert  
\leq \sum_{k = m}^{n-1} \vert f(q^{2l}) - f(q^{2(l+1)}) \vert \cdot \frac{\rho_q(l)}{\rho_q(k)} \\
& = \vert f(q^{2l}) - f(q^{2(l+1)}) \vert \cdot \rho_q(l) \cdot d_q( q^{2m},q^{2n}) .
\end{split}
\end{equation}
This shows that $f : X_q \to \CC$ is Lipschitz with $L_{d_q}(f) \leq L_{D_q}(f(A))$. The fact that equality is achieved is then a consequence of Definition \ref{d:defmetric}. Indeed, we obtain that 
\[
L_{D_q}(f(A)) = \vert f(q^{2l}) - f(q^{2(l+1)}) \vert \cdot \rho_q(l) 
= \frac{\vert f(q^{2l}) - f(q^{2(l+1)}) \vert}{d_q( q^{2l},q^{2(l+1)}) } \leq L_{d_q}(f) .
\qedhere
\] 
\end{proof}

The next lemma computes the Lipschitz semi-norms of general continuous functions on $X_q$ and provides information on the behaviour of the Lipschitz constants of a particularly interesting approximation.

\begin{lemma}\label{Lemma:Subseq}
For any $f\in C(X_q)$ one has 
\[
L_{d_q}(f)=\sup \lbrace \vert f(q^{2k})-f(q^{2(k+1)})\vert \cdot \rho_q(k) : k \in \NN_0 \rbrace.
\]
Moreover, if $f(0) = 0$ and $f$ is Lipschitz with respect to the metric $d_q$, then the sequence $\big\{ L_{d_q}(f \cdot \chi_{\{q^{2k} : k\leq n\}}) \big\}_{n = 0}^\infty$ is bounded.
\begin{proof}

We first notice that Definition \ref{d:defmetric} implies the inequality
\[
\sup \lbrace \vert f(q^{2k})-f(q^{2(k+1)})\vert \cdot \rho_q(k) : k \in \NN_0 \rbrace \leq L_{d_q}(f)
\]
(see also the proof of Lemma \ref{Lemma:Linearcomb} for more details).

We then claim that
\begin{equation}\label{eq:supest}
\frac{\vert f(x)-f(y) \vert }{d_q(x,y)}\leq 
\sup \lbrace \vert f(q^{2k}) -f(q^{2(k+1)}) \vert \cdot \rho_q(k) : k \in \NN_0  \rbrace
\end{equation}
whenever $x,y \in X_q \setminus \{0\}$ satisfy $x \neq y$. We have to be careful at this point since the inequality in \eqref{eq:supest} is not an immediate consequence of Definition \ref{d:defmetric} \, : the right hand side of our inequality only uses successive elements as exponents (i.e. $k$ and $k+1$) whereas $x = q^{2n}$ and $y = q^{2m}$ for some $n,m \in \NN_0$ without any further constraints (except for $n \neq m$). The inequality in \eqref{eq:supest} does however follow by an application of Lemma \ref{Lemma:Linearcomb} to a suitable restriction of $f$.

Thus, to establish the claimed identity, it only remains to be shown that the supremum in \eqref{eq:supest} is still an upper bound when $x=q^{2n}$ for some $n\in \NN_0$ and $y=0$. However, this follows immediately from the estimate in \eqref{eq:supest} together with continuity of the function $f$ and the metric $d_q$.

For the second part, we assume that $f$ is Lipschitz and that $f(0) = 0$. By Lemma \ref{Lemma:Linearcomb} it suffices to show that the sequence $\lbrace\vert f(q^{2n}) \vert \cdot \rho_q(n) \rbrace_{n = 0}^\infty$ is bounded. To this end, we first note that since $f$ is Lipschitz we may find a constant $C$ such that $\vert f(q^{2n}) \vert \leq C\cdot d_q(q^{2n},0)$ for all $n\in \NN_0$. It follows that
\[
\rho_q(n) \cdot \vert f(q^{2n}) \vert \leq C\cdot \sum_{k=n}^\infty \frac{\rho_q(n)}{\rho_q(k)}= C \cdot \sum_{k=0}^\infty q^k\frac{ \sqrt{1-q^{2(n+1)}}}{\sqrt{1-q^{2(k+n+1)}}} \leq C \cdot \sum_{k=0}^\infty q^k = \frac{C}{1-q}
\]
for all $n\in \NN_0$. This ends the proof of the lemma. \end{proof} 
\end{lemma}

The metric $d_q \colon X_q \times X_q \to [0,\infty)$ yields a Lipschitz algebra $C_{\Lip}(X_q) \subset C(X_q)$ and the semi-norm $L_{D_q} \colon I_q \to [0,\infty]$ yields an \emph{a priori different} Lipschitz algebra $\Lip_{D_q}(I_q) \subset I_q$. The Lipschitz algebras $C_{\Lip}(X_q)$ and $\Lip_{D_q}(I_q)$ agree with the domains of the semi-norms $L_{d_q}$ and $L_{D_q}$, respectively (recall that the domain consists of the elements where a semi-norm is finite). Moreover, the two unital commutative $C^*$-algebras $C(X_q)$ and $I_q$ are related by the $*$-isomorphisms $f \mapsto f(A)$. We are going to show that the $*$-isomorphism $f \mapsto f(A)$ restricts to a $*$-isomorphism $C_{\Lip}(X_q) \to \Lip_{D_q}(I_q)$ which is moreover isometric with respect to the semi-norms $L_{d_q}$ and $L_{D_q}$. 

Suppressing the identification $C(X_q) \cong I_q$ we have by now proved that the two semi-norms $L_{d_q}$ and $L_{D_q}$ agree on \emph{finite} linear combinations of the indicator functions $\chi_{\{q^{2k}\}}$, $k \in \NN_0$ (Lemma \ref{Lemma:Linearcomb}) and we have moreover succeeded in computing the semi-norm $L_{d_q} \colon C(X_q) \to [0,\infty]$ (Lemma \ref{Lemma:Subseq}). 

The passage from finite linear combinations of indicator functions to general Lipschitz elements is however quite subtle. To explain a bit what the subtle point is, we let $\mathcal{I}_q \subset I_q$ denote the smallest unital $*$-subalgebra containing all the projections $\chi_{\{q^{2k}\}}(A)$. Then even though $\mathcal{I}_q \subset I_q$ is norm-dense and the derivation $\partial : \Lip_{D_q}(I_q) \to B(H_+ \oplus H_-)$ is closed, it is not true that $\Lip_{D_q}(I_q)$ can be recovered by taking the closure of the restriction $\partial \colon \mathcal{I}_q \to B(H_+ \oplus H_-)$. In particular, for a general element $f(A) = f(0)+\sum_{k = 0}^\infty f(q^{2k}-f(0))\cdot \chi_{\{q^{2k}\}}(A) \in \Lip_{D_q}(I_q)$ we cannot a priori compute $\partial(f(A)) \in B(H_+ \oplus H_-)$ by using Lemma \ref{Lemma:Char} and applying the derivation $\partial$ term by term.

After these clarifications we are ready to state and prove the first main result of this section:

\begin{theorem}\label{Thm:Lip-identification}
The Lip-algebra of $I_q$ associated with the Dąbrowski-Sitarz spectral triple $(S_q^2,H_+ \oplus H_-,D_q)$ agrees with $\lbrace f(A) : f\in C_{\Lip}(X_q) \rbrace$, and for $f\in C_{\Lip}(X_q)$, we have $L_{D_q}(f(A))=L_{d_q}(f)$.
\end{theorem}

\begin{proof}
Let $f \in C(X_q)$ be given. 

Suppose first that $L_{D_q}(f(A)) < \infty$. For each $n \in \NN_0$ we define the projection $Q_n := \sum_{k = 0}^n \chi_{ \{q^{2k}\}}(A)$. Since $\partial_1$ is a derivation, we obtain from Lemma \ref{Lemma:Char} that
\[
\begin{split}
\partial_{1} (f(A)) Q_n & = \partial_1(f(A) Q_n )-f(A)\partial_1( Q_n) \\
& =  \sum_{k = 0}^{n-1} \big( f(q^{2k}) - f(q^{2(k+1)}) \big) \frac{1}{q^{2k} (1 - q^2)} \chi_{ \{q^{2k}\}}(A) \cdot b^* a^* .
\end{split}
\]
Following the proof of {\black Lemma \ref{Lemma:Linearcomb}} we then get that
\begin{equation}
\| \partial_{1} (f(A)) Q_n \| 
= \max\{ \vert f(q^{2k}) - f(q^{2(k+1)}) \vert \cdot \rho_q(k) : k \in \{0,1,\ldots,n-1\} \} 
\end{equation}
and hence (using that $Q_n$ is an orthogonal projection) we obtain the estimate
\begin{equation}\label{Eq:Proj1}
\begin{split}
\sup\{ \vert f(q^{2k}) - f(q^{2(k+1)}) \vert \cdot \rho_q(k) : k \in \NN_0 \}
& = \sup\{ \| \partial_{1} (f(A)) Q_n \|  : n \in \NN_0 \} \\ 
& \leq \| \partial_1( f(A))\|  .
\end{split}
\end{equation}
By Lemma \ref{Lemma:Subseq} this shows that $f$ is Lipschitz with respect to the metric $d_q$ and that
\[
L_{d_q}(f) \leq \| \partial_1(f(A)) \| .
\]
To prove that equality holds, we observe that by \cite[Theorem 6.2.17]{Timmermann}, 
\[
h(Q_n)= (1-q^2) \sum_{k=0}^n q^{2k}\underset{n\to \infty}{\To} 1,
\]
where $h$ denotes the Haar state on $SU_q(2)$. Since $h$ is faithful and $\lbrace Q_n \rbrace_{n = 0}^\infty$ is an increasing sequence of projections, $Q_n$ converges to the identity in the strong operator topology on $B(L^2(SU_q(2)))$, and hence also on $B(H_+)$. It now follows from \eqref{Eq:Proj1} and Lemma \ref{Lemma:Subseq} that for any $\xi$ in the unit ball of $H_+$, we have
\begin{align*}
\Vert \partial_1(f(A))\xi \Vert &= \lim_{n\to \infty} \Vert \partial_1(f(A))Q_n\xi \Vert   
\leq \sup\{ \|\del_1(f(A))Q_n \| : n \in \NN_0 \}\\
&= L_{d_q}( f) .
\end{align*}
and hence that $\Vert \partial_1(f(A))\Vert = L_{d_q}(f)$. Since we moreover have the identities 
\[
\Vert \partial_2(f(A))\Vert = \Vert \partial_1( \bar{f}(A)) \Vert = L_{d_q}(f)
\]
we may conclude that $L_{D_q}(f(A)) = L_{d_q}(f)$.

Suppose next that $f\in C(X_q)$ is Lipschitz with respect to the metric $d_q$. Since subtracting a constant changes neither the Lipschitz constant of $f$ nor $L_{D_q}(f(A))$, we may, without loss of generality, assume that $f(0)=0$. For each $n \in \NN_0$ define the function $f_n := f \cdot \chi_{\{q^{2k} : k\leq n\}}$. By Lemma \ref{Lemma:Subseq}, the sequence $\lbrace L_{d_q}(f_n)\rbrace_{n=0}^\infty$ is then bounded and moreover $f_n(A)$ converges to $f(A)$ in operator norm. 

Hence, since $L_{D_q}(f_n(A))=L_{d_q}(f_n)$ by Lemma \ref{Lemma:Linearcomb}, we obtain by lower semi-continuity of $L_{D_q} : I_q \to [0,\infty]$ that
\[
L_{D_q}(f(A)) \leq \sup\lbrace L_{D_q}(f_n(A)) : n \in \NN_0 \rbrace < \infty .
\]
This shows that $f(A) \in \Lip_{D_q}(I_q)$ and this ends the proof of the theorem.
\end{proof}

Theorem \ref{Theorem:Lip-Lip} now follows easily:
\begin{proof}[Proof of Theorem \ref{Theorem:Lip-Lip}]
The metric $d_q'$ on $X_q$ induced by $L_{D_q}$ is by definition given by
\begin{align*}
d_q'(x,y):=\sup\{|f(x)- f(y)| : f \in C(X_q), L_{D_q}(f(A))\leq 1\}.
\end{align*}
However, by Theorem \ref{Thm:Lip-identification} we have
\begin{align*}
d_q(x,y)&= \sup\{ |f(x)-f(y)| : f\in C(X_q), L_{d_q}(f)\leq 1\}\\
&=\sup\{|f(x)-f(y)| : f\in C(X_q), L_{D_q}(f(A))\leq 1\},
\end{align*}
and hence the two metrics agree.
\end{proof}
In the following, we will consider the behaviour of $(X_q,d_q)$ with respect to the Gromov-Hausdorff metric, and provide a proof of Theorem \ref{Theorem:Cont}. To this end, we first establish a preliminary result about the diameter of $X_q$:

\begin{lemma}
\label{Lemma:Diam}
It holds that $\lim_{q\to 1} d_q(0,1)=\pi$.
\begin{proof}
Observe that the function $\frac{1}{\rho_q}\colon x \mapsto (1-q^2)\frac{q^x}{\sqrt{1-q^{2(x+1)}}}$ is positive and decreasing on $(-1,\infty)$.  This yields the estimates
 \begin{equation}\label{Eq:IntegralRule}
  \int_1^\infty \frac{1}{\rho_q(x)}\,dx \leq  \sum_{k=0}^\infty \frac{1}{\rho_q(k)} \leq  \int_0^\infty \frac{1}{\rho_q(x)}\,dx .
 \end{equation}
Furthermore, it can be verified that $F(x):= \frac{1-q^2}{q\ln(q)}\arcsin(q^{x+1})$ is an antiderivative of $\frac{1}{\rho_q(x)}$ and  $\lim_{x\to \infty} F(x)=0$. We therefore obtain the inequalities
\[
-\frac{1-q^2}{q\ln(q)}\arcsin(q^2) \leq d_q(0,1) 
\leq - \frac{1-q^2}{q\ln(q)}\arcsin(q) .
\]
Since $\lim_{q\to 1}\frac{1-q^2}{q\ln(q)}=-2$ and $\arcsin(1) = \frac{\pi}{2}$ we may conclude that $\lim_{q \to 1} d_q(0,1) = \pi$.
\end{proof}
\end{lemma}

\begin{proof}[Proof of Theorem \ref{Theorem:Cont}]
For each $q \in (0,1)$, we consider the isometric embedding $\iota_q\colon X_q \to \RR$ given by $\iota_q(x)=d_q(1,x)-\tfrac{\pi}{2}$. 

We start by proving continuity at a fixed $q_0\in (0,1)$. Let $\varps>0$ be given. Choose a $\delta_0 > 0$ such that $J := [q_0 - \delta_0, q_0 + \delta_0] \subset (0,1)$. From the estimate in \eqref{eq:estrho} we obtain that
\[
\sum_{k = 0}^\infty \sup \left\lbrace \frac{1}{\rho_q(k)} : q \in J \right\rbrace  
\leq \sum_{k = 0}^\infty \sup \lbrace q^k \sqrt{1 - q^2} : q \in J \rbrace  
\leq \sum_{k = 0}^\infty (q_0 + \delta_0)^k < \infty .
\]
We may therefore choose an $n_0 \in \NN_0$ such that
\begin{equation}\label{eq:contGHI}
\sum_{k=n_0}^{\infty} \frac{1}{\rho_q(k)}< \frac{\varps}{3}
\end{equation}
for all $q \in J = [q_0 - \delta_0, q_0 + \delta_0]$. Now, for each $k \in \NN_0$, the function $q \mapsto \sum_{k = 0}^{n_0 - 1}\frac{1}{\rho_q(k)}$ is continuous and we may thus choose a $\delta \in (0,\delta_0)$ such that 
\begin{equation}\label{eq:contGHII}
\left|\sum_{k=0}^{m-1} \frac{1}{\rho_q(k)} -  \sum_{k=0}^{m-1} \frac{1}{\rho_{q_0}(k)}\right|< \frac{\varps}{3} 
\end{equation}
for all $m \in \{1,\ldots,n_0\}$ and all $q \in (q_0 - \delta,q_0 + \delta)$.

Let now $q \in (q_0 - \delta,q_0 + \delta) \subset J$ be given. It then follows immediately from \eqref{eq:contGHII} that 
\[
|\iota_q(q^{2m}) - \iota_{q_0}(q_0^{2m})| < \frac{\varps}{3} < \varps
\]
for all $m \in \{1,\dots, n_0\}$. Moreover, for $m > n_0$ we apply \eqref{eq:contGHI} and \eqref{eq:contGHII} to estimate that
\[
\begin{split}
|\iota_q(q^{2m})- \iota_{q_0}(q_0^{2m})| &= \left| \sum_{k=0}^{n_0-1}\frac{1}{\rho_q(k)} +\sum_{k=n_0}^{m-1} \frac{1}{\rho_q(k)} - \sum_{k=0}^{n_0-1}\frac{1}{\rho_{q_0}(k)} - \sum_{k=n_0}^{m-1} \frac{1}{\rho_{q_0}(k)} \right| \\
&\leq | \iota_q( q^{2n_0}) - \iota_{q_0}(q_0^{2n_0}) | 
+ \sum_{k=n_0}^{\infty} \frac{1}{\rho_q(k)} + \sum_{k=n_0}^{\infty} \frac{1}{\rho_{q_0}(k)}\\
&<  \varps .
\end{split}
\]
A similar argument also shows that $|\iota_q(0) - \iota_{q_0}(0)| < \varps$. We conclude that 
\[
\operatorname{dist}_H(\iota_q(X_q),\iota_{q_0}(X_{q_0})) \leq \epsilon
\]
and hence that $(0,1) \ni q \mapsto (X_q,d_q)$ varies continuously in Gromov-Hausdorff distance.

For convergence, it suffices to show that the Hausdorff distance between $\iota_q(X_q)$ and $\left[-\tfrac{\pi}{2},\tfrac{\pi}{2}\right]$ converges to $0$ as $q\to 1$. To this end, let $\epsilon>0$ be arbitrary. By Lemma \ref{Lemma:Diam}, we may find a $q_1\in (0,1)$ such that for any $q\in (q_1,1)$, we have $|\iota_q(0)-\tfrac{\pi}{2}|<\epsilon$. Moreover, since $-\frac{\pi}{2}\leq \iota_q(x)\leq \iota_q(0)$ for all $x\in X_q$, it follows that for every $x\in X_q$ there exists a $y\in  \left[-\tfrac{\pi}{2},\tfrac{\pi}{2}\right]$ with $|\iota_q(x)-y|<\epsilon$. 
It remains to be shown that we can find a $q_2\in (0,1)$ such that given any $y\in \left[-\tfrac{\pi}{2},\tfrac{\pi}{2}\right]$ and any $q\in (q_2,1)$, we can find $x\in X_q$ such that $|y-\iota_q(x)|<\epsilon$. Since $\frac{1}{\rho_q(0)}=\sqrt{1-q^2} \underset{q\to 1}{\To} 0$ and $d_q(0,1) \underset{q\to 1}{\To} \pi$ by Lemma \ref{Lemma:Diam}
we can find a $q_2 \in (0,1)$ such that $\frac{1}{\rho_q(0)}<\epsilon$ and $\vert \iota_q(0)-\tfrac{\pi}{2}\vert <\tfrac{\epsilon}{2}$ for all $q\in(q_2,1)$. Let now $q \in (q_2,1)$ be given. It follows that $|y-\iota_q(0)|<\epsilon$ for $y\in \left(\tfrac{\pi}{2}-\tfrac{\epsilon}{2},\tfrac{\pi}{2}\right]$. On the other hand, we may for each $y\in \left[-\tfrac{\pi}{2},\tfrac{\pi}{2}-\tfrac{\epsilon}{2}\right]$ find an $n\in \NN_0$ such that $y\in [\iota_q(q^{2n}),\iota_q(q^{2(n+1)})]$ and consequently 
\[
\left|y-\iota_q(q^{2n})\right| \leq \left|\iota_q(q^{2n})-\iota_q(q^{2(n+1)})\right| = \frac{1}{\rho_q(n)}\leq \frac{1}{\rho_q(0)}<\epsilon . \qedhere
\]
\end{proof}
\begin{remark}
\label{Rem:Dist}
As stated in the introduction, Theorem \ref{Theorem:Cont} also applies if we replace the classical Gromov-Hausdorff distance with respectively the quantum Gromov-Hausdorff distance of Rieffel \cite{Rieffel-distance} or Latrémolière's propinquity. To see this, note that by \cite[Corollary 6.4]{Latremoliere} the former is dominated by two times the latter and by \cite[Theorem 6.6]{Latremoliere}, propinquity is dominated by the classical Gromov-Hausdorff distance on the class of compact metric spaces, and hence the convergence and continuity are also obtained for these distances. 
\end{remark}
\def\cprime{$'$} \def\cprime{$'$}

\end{document}